\documentclass[a4paper, 11pt]{amsart}
\usepackage{a4wide}

\parskip=\smallskipamount

\newtheorem{theorem}{Theorem}[section]

\newtheorem{corollary}[theorem]{Corollary}
\newtheorem{proposition}[theorem]{Proposition}

\theoremstyle{definition}
\newtheorem{definition}[theorem]{Definition}

\newcommand{\C}{\mathbb{C}}

\newcommand{\N}{\mathbb{N}}

\begin{document}

\title{The complement of the closed unit ball in $\mathbb C^3$ is not subelliptic}

\author{Rafael B. Andrist}
\address{Bergische Universit\"at Wuppertal, Fachbereich C,
Gau{\ss}str. 20, D-42119 Wuppertal, Germany}
\email{rafael.andrist@math.uni-wuppertal.de}

\author{Erlend Forn{\ae}ss Wold}
\address{Matematisk Institutt, Universitetet i Oslo, Postboks 1053 Blindern, 0316 Oslo, Norway}
\email{erlendfw@math.uio.no}

\begin{abstract}
In this short note we show that $\mathbb C^n\setminus\overline{\mathbb B^n}$ is not subelliptic for $n \geq 3$. This is done by proving a Hartogs type extension theorem for holomorphic vector bundles bundles.
\end{abstract}

\maketitle

\section{Introduction}

Gromov \cite{Gromov-elliptic} introduced 1989 in his seminal paper the notion of an elliptic manifold and proved an Oka principle for holomorphic sections of elliptic bundles, generalizing previous work of Grauert \cite{Grauert-Lie} for complex Lie groups. The (basic) Oka principle of a complex manifold $X$ is that every continuous map $Y \to X$ from a Stein space $Y$ is homotopic to a holomorphic map.

In Oka theory exist many interesting classes of complex manifolds with weaker properties than ellipticity. However, all the known inclusion relations between these classes are yet not known to be proper inclusions. Following the book of Forstneri\v{c} \cite[Chap. 5]{Forstneric-book} we want to mention in particular the following classes and then prove that at least one of these inclusions has to be proper.

\begin{definition}
A \emph{spray} on a complex manifold $X$ is a triple $(E, \pi, s)$ consisting of a holomorphic vector bundle $\pi : E \to X$ and a holomorphic map $s: E \to X$  such that for each point $x \in X$ we have $s(0_x) = x$.

The spray $(E, \pi, s)$ is said to be \emph{dominating} if for every point $x \in X$ we have
\[
\mathrm{d}_{0_x} s (E_{x}) = \mathrm{T}_x X
\]
A complex manifold is called \emph{elliptic} if it admits a dominating spray.
\end{definition}

A weaker notion, subellipticity, was introduced later by Forstneri\v{c} \cite{Forstneric-subelliptic} where he proved the Oka principle for subelliptic manifolds:

\begin{definition}
A finite family of sprays $(E_j, \pi_j, s_j), j = 1, \dots, m,$ on $X$
is called \emph{dominating} if for every point $x \in X$ we have
\begin{equation}\label{dominating}
\mathrm{d}_{0_x} s_1 (E_{1,x}) + \dots + \mathrm{d}_{0_x} s_m (E_{m,x}) = \mathrm{T}_x X
\end{equation}
A complex manifold $X$ is called \emph{subelliptic} if it admits a finite dominating family of sprays.
\end{definition}

It is immediately clear from the definition that an elliptic manifold is subelliptic. However, the Oka principle holds under even weaker conditions. Forstneri\v{c} \cite{Forstneric-Oka1, Forstneric-Oka2, Forstneric-Oka3} showed that the following condition (CAP) is equivalent for a manifold to satisfy the Oka principle (and versions of the Oka principle with interpolation and approximation), hence justifing the name Oka manifold:

\begin{definition}
A complex manifold is said to satisfy the \emph{convex approximation property} (CAP) if
for on any compact convex set $K \subset \C^n, \; n \in \N,$ every holomorphic map $f : K \to X$ can be approximated, uniformly on $K$, by entire holomorphic maps $\C^n \to X$.  If this approximation property holds for a fixed $n \in \N$, then $X$ is said to satisfy $\mathrm{CAP}_n$.
A manifold satisfying CAP is called an \emph{Oka manifold}.
\end{definition}

A subelliptic manifold is always Oka. Whether an Oka manifold is (sub-)elliptic, is on the other hand not known -- this implication holds however under the extra assumption that it is Stein.

We want to mention also these weaker properties:

\begin{definition}
A complex manifold $X$ of dimension $n$ is called \emph{dominable} if there exists a point $x_0 \in X$ and a holomorphic map $f: \C^n \to X$ with $f(0) = x_0$ and $\mathrm{rank}\, \mathrm{d}_0 f = n$.
\end{definition}

\begin{definition}
A complex manifold $X$ of dimension $n$ is called \emph{strongly dominable} if for every  point $x_0 \in X$ there exists a holomorphic map $f: \C^n \to X$ with $f(0) = x_0$ and $\mathrm{rank}\, \mathrm{d}_0 f = n$.
\end{definition}

Summarizing the previous, the following inclusions are known:
\[
\mathrm{elliptic} \subseteq \mathrm{subelliptic} \subseteq \mathrm{Oka} \subseteq \mathrm{strongly \, dominable} \subseteq \mathrm{dominable}
\]

There are several known candidates to prove that one of this inclusions is proper.
We will present here an example of complex manifold which is not subelliptic, but strongly dominable. Whether or not it is Oka remains an open question.

\begin{theorem}\label{main}
For $n\geq 3$ the manifold $X:=\mathbb C^n\setminus\overline{\mathbb B^n}$ 
is not subelliptic.
\end{theorem}

The proof of this relies on a Hartogs type extension result for holomorphic 
vector bundles in dimensions $n\geq 3$.   This extension result is not 
valid in dimension 2, so the question as to whether $\mathbb C^2\setminus\overline{\mathbb B}^2$
is elliptic or sub-elliptic remains open.  

\medskip

Recently, Forstneri\v{c} and Ritter \cite{ForstnericRitter} have proved that $X = \mathbb C^n \setminus \overline{\mathbb B^n}$
satisfies $\mathrm{CAP}_m$ for all $m<n$, but it remains an open question whether $X$ is an Oka-manifold.

\section{Proof of Theorem \ref{main}}

Assume that $E\rightarrow X$ is a holomorphic vector bundle with a 
spray $s:E\rightarrow X$. 
By Corollary \ref{extension2} there 
exists a point $p\in b\mathbb B^n$ and a (small) open ball $B$ centered at $p$
such that $E|_B$ is the trivial bundle $B\times\mathbb C^r$.  We use 
$(z,w)$ as coordinates on $B\times\mathbb C^r$.  On $(B\setminus\overline{\mathbb B}^n)\times\mathbb C^r$
we may write the spray $s=(s_1,...,s_n)$ as 
$$
s_j(z,w)=\sum_{\alpha} g^j_\alpha(z)w^\alpha, 
$$
where $\alpha$ is a multi-index. By the Hartogs extension theorem for 
the functions $g^j_\alpha$ we see that the spray extends to some 
neighborhood of $p$, and so by possibly having to choose 
a smaller $B$ we assume that $s$ extends to $B\times\mathbb C^r$.

\medskip

Given a finite number of sprays $(E_j,\pi_j,s_j)$ we may repeat the
argument for all of them simultaneously to find a ball $B$ around 
a point $p\in b\mathbb B^n$ to which the bundles and sprays 
extend.  We claim that the set of points $Z\subset B$ where 
the family of sprays is not dominating is an analytic set. 
To see this, choose holomorphic sections $\mathcal B_j=\{f^j_1,\dots,f^j_{k_j}\}$
for $j=1,\dots,m$, where $k_j=\dim(E_j)$ such that $\mathcal B_{j,z}$
forms a basis for $E_{j,z}$ for all $z\in B$.
Let $g^j_s(z)=\mathrm{d}_{0_z}s_j(f^j_s(z))$, 
and let $Z\subset B$ be the set of points in $B$ where the dimension 
of the span of all the vectors $g^j_s$  is less than $n$. 
The set $Z$ has to be a proper analytic subset of $B$ and 
so there exists a point $q \in b\mathbb B^n$ such that 
one of the sprays, say $s_1$ satisfies $\mathrm{d}_{0_q}s_1(E_{1,z})\cap\mathbb B^n\neq\emptyset$, 
which implies that $s_{1,z}(E_{1,z})\cap\mathbb B^n\neq\emptyset$.  
But then $s_1(E_{1,w})\cap\mathbb B^n\neq\emptyset$ for 
points $w$ close to $q$ which contradicts the assumption 
that $s_1$ is a spray into $X$.   
$\hfill\square$

The argument given relies on the following extension result for locally free sheaves by Siu \cite{Siu}.

\medskip

\begin{definition}
For $0<a<b\in\mathbb R^N$ we let $G^N(a,b)$ denote the set 
\begin{gather*}
G^N(a,b) = 
 \{z \in \C^N: |z_j|<b_j \mbox{ for } j = 1,\dots,N \mbox{ and } |z_j| > a_j \mbox{ for some } 1 \leq j \leq N \}.
\end{gather*}
\end{definition}

\begin{proposition}(Siu, \cite{Siu}, page 144)\label{unique}
Suppose $0\leq a\leq a'<b\in\mathbb R^N$ and $D$ is an 
open subset of $\mathbb C^n$.  Suppose $S_i$ is a coherent
analytic sheaf on $D\times G^N(a,b)$ such that $S_i=S_i^{[n]}$ $(i=1,2)$.
If $\varphi:S_1\rightarrow S_2$ is a sheaf-isomorphism on $D\times G^N(a',b)$
then $\varphi$ can be uniquely extended to a sheaf isomorphism $S_1\rightarrow S_2$ on $D\times G^N(a,b)$. 
\end{proposition}

\begin{theorem}(Siu, \cite{Siu}, page 225)\label{extension1}
Suppose $0\leq a<b\in\mathbb R^2$, $D$ is a domain in $\mathbb C^n$, 
and $S$ is a locally free sheaf of rank $r$ on $D\times G^2(a,b)$.  Suppose 
$A$ is a thick set in $D$ and, for every $t\in A$ $S(t)$ can be extended 
to a coherent analytic sheaf on $\{t\}\times\triangle^2(b)$.  Then 
$S$ can be extended uniquely to a coherent analytic sheaf $\tilde S$ on 
$D\times\triangle^2(b)$ satisfying $\tilde S^{[n]}=\tilde S$.
\end{theorem}

\begin{corollary}\label{extension2}
Let  $\pi:E\rightarrow\mathbb C^n\setminus\overline{\mathbb B}^n$ be 
a holomorphic vector bundle of rank $r$. 
Then there exists a finite set of points $P\subset b\mathbb B^n$ such that 
for any point $p\in b\mathbb B^n\setminus P$, there exists 
an open neighborhood $U_p$ of $p$ such that $E$ extends 
to a holomorphic vector bundle on $U_p$.
\end{corollary}

\begin{proof}
Let $p=(1,0,\cdot\cdot\cdot,0)$.   Let $B_\epsilon$ be the ball
of radius $\epsilon$ centered at $p$ in the plane $\{z_{n-1}=z_n=0\}$, 
let $a=(\sqrt{1-(1-\epsilon)^2},\sqrt{1-(1-\epsilon)^2})$, and let $b=(1,1)$.
By Theorem \ref{extension1} we get an extension of 
$E$ from $B_\epsilon\times G^2(a,b)$ to $B_\epsilon\times\triangle^2(b)$
as a coherent analytic sheaf $\tilde E$.  By \cite{Grauert} $\tilde E$
is a vector bundle outside an analytic set $Z$, i.e., a finite 
set of points, since by Proposition \ref{unique}  Z cannot project 
to the set $\{|z_1|^2+\cdot\cdot\cdot +|z_{n-3}|^2>1\}$.
So perhaps after having to choose a different point $p$
and a smaller $\epsilon$, and referring to Proposition \ref{unique},
we may assume that $\tilde E$ is a vector bundle on $B_\epsilon\times\triangle^2(b)$.
Hence $\tilde E$ is trivial, and we may find holomorphic 
sections $s=(s_1,...,s_r)$ of the original bundle $E\rightarrow B_\epsilon\times G^2(a,b)$
generating $E$.    Since $s$ represents an isomorphism between 
$E$ and the trivial bundle, we may use Proposition \ref{unique}
on different Hartogs figures $B_{\delta(z)}(z)\times G^2(a(z),b), z\in B_\epsilon$
to extend $s$ holomorphically to $E\rightarrow(B_\epsilon\times\triangle^2(b))\setminus\overline{\mathbb B}^3$
(note that $\{z_{n-1}=z_n=0\}$ is removable for a vector bundle isomorphism). 
The extension is generating outside a hypersurface, hence everywhere, 
and so defines an isomorphism with the trivial bundle.  
\end{proof}

\bibliographystyle{amsplain}

\end{document}